\documentclass[reqno]{amsart}

\usepackage{graphicx,amssymb,mathrsfs,amsmath,color,fancyhdr}
\usepackage[all]{xy}
\newtheorem{theorem}{Theorem}[section]
\newtheorem{lemma}[theorem]{Lemma}

\theoremstyle{definition}
\newtheorem{definition}[theorem]{Definition}

\newtheorem{prob}{Problem}

\theoremstyle{remark}

\numberwithin{equation}{section}

\begin{document}

\title[Cyclic sieving on noncrossing (1,2)-configurations]
{Cyclic sieving on noncrossing (1,2)-configurations}

\author{Chuyi Zeng}
\address{School of Mathematics (Zhuhai), Sun Yat-sen University, Zhuhai 519082, Guangdong, P.R. China}
\email{zengchy28@mail2.sysu.edu.cn}
\author{Shiwen Zhang}
\address{School of Mathematics (Zhuhai), Sun Yat-sen University, Zhuhai 519082, Guangdong, P.R. China}
\email{zhangshw9@mail2.sysu.edu.cn}



\begin{abstract}
Verifying a suspicion of Propp and Reiner concerning the cyclic sieving phenomenon (CSP), M. Thiel introduced a Catalan object called noncrossing $(1,2)$-configurations (denoted by $X_n$), which is a class of set partitions of $[n-1]$. More precisely, Thiel proved that, with a natural action of the cyclic
group $C_{n-1}$ on $X_n$, the triple $\left(X_n,C_{n-1},\text{Cat}_n(q)\right)$ exhibits the CSP, where $\text{Cat}_n(q):=\frac{1}{[n+1]_q}\begin{bmatrix}
 2n\\
n
\end{bmatrix}_q$ is MacMahon's $q$-Catalan number. Recently, in a study of the fermionic diagonal coinvariant
ring $FDR_n$, J. Kim found a combinatorial basis for $FDR_n$ indexed by $X_n$. In this paper, we continue to study $X_n$ and obtain the following results: 
\begin{enumerate}
    \item We define a statistic $cwt$ on $X_n$ whose generating function is $\text{Cat}_n(q)$, which answers a problem of Thiel.
    \item We show that $\text{Cat}_n(q)$ is equivalent to $$\sum_{\substack{k,x,y\\2k+x+y=n-1}}\begin{bmatrix}
			n-1\\
			2k,x,y
		\end{bmatrix}_q\text{Cat}_k
  (q)q^{k+\binom{x}{2}+\binom{y}{2}+\binom{n}{2}}$$
  modulo $q^{n-1}-1$, which answers a problem of Kim. As mentioned by Kim, this result leads to a representation theoretic proof of the above cyclic sieving result of Thiel.
    \item We consider the dihedral sieving, a generalization of the CSP, which was recently introduced by Rao and Suk. Under a natural action of the dihedral group $I_2(n-1)$ (for even $n$), we prove a dihedral sieving result on $X_n$. 
\end{enumerate}

\end{abstract}

\maketitle{}

\section{Introduction}
The {\it cyclic sieving phenomenon} (CSP) has been a very active research topic in combinatorics since its introduction by Reiner, Stanton and White \cite{RSW} in 2004.  Let $X$ be a finite set carrying an action of a cyclic group $C_n=\{1,c,c^2,\dots, c^{n-1}\}$ and $X(q)$ a polynomial in $q$ with nonnegative integeral coefficients. We say that the triple $(X,C_n,X(q))$ exhibits the cyclic sieving phenomenon
if $X(1)=\vert X\vert$ and for every positive integer $d$, we have
$$|\{x\in X: c^d(x)=x\}|=X(\zeta_n^d),$$
    where $\zeta_n=e^{\frac{2\pi i}{n}}$. The CSP generalized Stembridge's $q=-1$ phenomenon \cite{Stem,Stembridge1994OnMR,MR1387685}. We refer to \cite{RSW2} and \cite{Sa} for some examples of the CSP.

Many studies of the CSP are concerned with Catalan objects, i.e., objects counted by the famous Catalan numbers $\text{Cat}_n:=\frac{1}{n+1}\binom{2n}{n}$. We refer to \cite{Stan2} for more than 200 interpretations of the Catalan numbers.  To consider the CSP for Catalan objects, in many cases, the polynomial $X(q)$ is choosen to be MacMahon's $q$-Catalan number:
$$\text{Cat}_n(q):=\frac{1}{[n+1]_q}\begin{bmatrix}
 2n\\
n
\end{bmatrix}_q,$$ 
which is defined by the $q$-analogues:
$$[n]_q:=1+q+q^2+\dots+q^{n-1},$$
$$[n]!_q:=[n]_q[n-1]_q\dots[2]_q[1]_q,$$
and the $q$-binomial coefficient
$$\begin{bmatrix}
 n\\
k
\end{bmatrix}_q:=\frac{[n]!_q}{[k]!_q[n-k]!_q}.$$ Note that the $q$-binomial coefficient is a polynomial with nonnegative integral coefficients \cite{Macdonald}. We refer to \cite{ALPU} for various cyclic sieving results for Catalan objects. 

Computational experiments by Propp and Reiner suggested that substituting an $(n-1)$-th root of unity into $\text{Cat}_n(q)$ always yields a positive integer. So they suspected that there exists a set $X_n$ with cardinality $|X_n|=\text{Cat}_n$, under the action of a cyclic group $C_{n-1}$ of order $n-1$, such that the triple $(X_n,C_{n-1},\text{Cat}_n(q))$ exhibits the CSP. In 2017, Thiel \cite{Thiel} verified their suspicion by constructing a Catalan object $X_n$ which is defined as follows.

\begin{definition}
    Let $n\ge 1$ and $[n]:=\{1,2,\dots,n\}$. 
    We call a subset of $[n]$
    \begin{itemize}
        \item a ball, if it has cardinality $1$;
        \item an arc, if it has cardinality $2$.
    \end{itemize}
    A \textbf{(1,2)-configuration} of $[n]$ is a set of pairwise disjoint balls and arcs.
    Say that a (1,2)-configuration 
    is \textbf{noncrossing} if it has no arcs $\{i_1,j_1\}$ and $\{i_2,j_2\}$ satisfying $i_1<i_2<j_1<j_2$. 
\end{definition}

\begin{figure}
    \centering
    \includegraphics{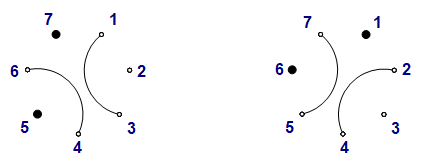}
    \caption{A noncrossing (1,2)-configuration\\ $x=\{\{1,3\},\{4,6\},\{5\},\{7\}\}$ (left) and its rotation $r(x)$ (right).}
    \label{fig:enter-label}
\end{figure}

We denote by $X_n$ the set of all noncrossing (1,2)-configurations of $[n-1]$. For example, $x=\{\{1,3\},\{4,6\},\{5\},\{7\}\}$ is an element in $X_8$. Thiel proved that the cardinality of $X_n$ is $\text{Cat}_n$ \cite[Theorem 2.1]{Thiel}.  There is a natural action of the cyclic group $C_{n-1}=\langle r\rangle$ on $X_n$: 
\begin{align*}
    r:[n-1]&\to[n-1],\\
  i&\mapsto i+1,\ \text{if}\ i\neq n-1 \\
  n-1&\mapsto 1.
\end{align*}
An example of this action is presented in Fig.\ref{fig:enter-label}. By direct computation, Thiel proved that the triple $(X_n,C_{n-1},\text{Cat}_n(q))$ exhibits the CSP and he proposed the following problem.

\begin{prob}\cite[Problem]{Thiel}\label{Thiel'sproblem}
Find a nice statistic on the set $X_n$ of noncrossing (1,2)-configurations whose generating function is $\text{Cat}_n(q)$.
\end{prob}

In Section 2, we answer Problem  \ref{Thiel'sproblem}. Our proof is based on an interpretation of $\text{Cat}_n(q)$ involving the major index statistic, which was obtained by MacMahon \cite{macmahon2001combinatory}.

Recently, the set $X_n$ also appeared in a study of the fermionic diagonal coinvariant ring, which was introduced by Jongwon Kim and Rhoades \cite{JKR1}. Let $\Theta_n=(\theta_1,\ldots,\theta_n)$ and $\Xi_n=(\xi_1,\ldots,\xi_n)$ be two sets of $n$ anticommuting variables. Denote by $\wedge\{\Theta_n,\Xi_n\}:=\wedge\{\theta_1,\ldots,\theta_n,\xi_1,\ldots,\xi_n\}$ the exterior algebra generated by these symbols over $\mathbb C$. Consider an action of the symmetric group $\mathfrak S_n$ on $\wedge\{\Theta_n,\Xi_n\}$ by permuting the indices in the following way
$\sigma\cdot \theta_i=\theta_{\sigma(i)}$ and $\sigma\cdot \xi_i=\xi_{\sigma(i)}$, where $\sigma\in \mathfrak S_n$ and $1\le i\le n$. Let $\wedge\{\Theta_n,\Xi_n\}^{\mathfrak S_n}_+$ denote the subspace of $\mathfrak S_n$-invariants with vanishing constant term. Then the fermionic diagonal coinvariant
ring is defined as
$$FDR_n:=\wedge\{\Theta_n,\Xi_n\}/\wedge\{\Theta_n,\Xi_n\}^{\mathfrak S_n}_+.$$
There is a bigraded structure on $FDR_n$ (see \cite{JKR2} for details). In \cite{JKR2}, Kim and Rhoades gave a combinatorial basis for the maximal degree components of $FDR_n$. This basis is indexed by certain set partitions for which the $\mathfrak S_n$-action is given by a skein action on noncrossing partitions first described by Rhoades \cite{Rh}. In \cite{kim2023combinatorial}, Kim generalized the above result for the entire $FDR_n$, where the combinatorial basis is considered under an action of $\mathfrak S_{n-1}\subset\mathfrak S_n$. Interestingly, this basis is indexed by the set of all noncrossing (1,2)-configurations of $[n-1]$, i.e., the set $X_n$.
Based on a theorem of Springer on regular elements \cite{springer1974regular}, Kim \cite[Theorem 6.1]{kim2023combinatorial} proved that the triple 
$$\left(X_n,C_{n-1},q^{\binom{n}{2}}\text{\bf fd}((FDR_n)_{i+j=n-1}))\right)$$ 
exhibits the CSP, where $\text{\bf fd}((FDR_n)_{i+j=n-1})$ is the fake degree of  $(FDR_n)_{i+j=n-1}$ (see \cite[Section 2]{kim2023combinatorial} for details) and in particular
$$\text{\bf fd}((FDR_n)_{i+j=n-1})=\sum_{\substack{k,x,y\\2k+x+y=n-1}}\begin{bmatrix}
			n-1\\
			2k,x,y
		\end{bmatrix}_q\text{Cat}_k
  (q)q^{k+\binom{x}{2}+\binom{y}{2}}.$$ Combining the above two cyclic sieving results on $X_n$, Kim proposed the following problem.

\begin{prob}\cite[Problem 6.3]{kim2023combinatorial}\label{prob2}
 Is there a direct computational proof that $$\text{Cat}_n(q)\equiv q^{\binom{n}{2}}\text{\bf fd}((FDR_n)_{i+j=n-1}))\ (\text{mod}\ q^{n-1}-1)?$$ 
\end{prob}
In Section 3, we answer problem \ref{prob2}. As mentioned by Kim, a solution of Problem \ref{prob2} leads to a representation theoretic proof of the above cyclic sieving result of Thiel.


Finally, we consider a generalization of the CSP which is called the dihedral sieving \cite{MR4064826}. Note that the CSP had been generalized for non-cyclic abelian groups; see, e.g., \cite{MR2418296,MR2837599}.
To generalized the CSP for all finite groups, we need to introduce an equivalent interpretation of the CSP in terms of representation theory. 

\begin{definition}\cite[Definition 2.2]{MR4064826}
    Let $G$ be a finite group and $A$ the set of isomorphism classes of finite-dimensional $G$-representations over $\mathbb C$. Let $\mathbb Z[A]$ be the polynomial ring over $\mathbb Z$ freely generated by $A$. Let $I$ (resp. $J$) be the ideal of $\mathbb Z[A]$ generated by the elements $[U\oplus V]-([U]+[V])$ (resp. $\{[U\otimes V]-([U][V])\}$), where $U$ and $V$ are finite-dimensional $G$-representations and $[U]$ (resp. $[V]$) denotes the isomorphism class of $U$ (resp. $V$). Then the {\it representation ring} $\text{Rep}(G)$ with coefficients in $\mathbb Z$ is defined as
$$\text{Rep}(G)=\mathbb Z[A]/(I+J).$$ 
\end{definition}

In other words, let $\text{Irr}(G)$ be the set of isomorphism classes of irreducible finite-dimensional $\mathbb C$-representations of $G$, the representation ring $\text{Rep}(G)$ is just the free abelian group $\mathbb Z[\text{Irr}(G)]$ with the multiplication defined by $[U]\cdot [V]=\sum_i[V_i]$ if $U\otimes V=\oplus_i V_i$ (\cite[Definition 2.1]{MR4459983}). 

For a cyclic group $C_n=\langle c \rangle$ of order $n$, the one-dimensional representation $\rho$ of $C_n$ defined by $c\mapsto e^{\frac{2\pi i}{n}}$ is a generator of $\text{Rep}(C_n)$. Then we have the following equivalent formulation \cite[Proposition 2.1]{RSW} of the CSP: The triple $(X,C_n,X(q))$ exhibits the CSP if and only if $[\mathbb{C}[X]]=[X(\rho)]$ in $\text{Rep}(C_n)$. Here, $\mathbb{C}[X]$ is the permutation representation of $C_n$ induced by the action of $C_n$ on $X$, $X(\rho)=a_m\rho^{\otimes m}\oplus\cdots \oplus a_1\rho\oplus a_01_{C_n}$ is a representation of $C_n$ if $X(q)=\sum_{i=0}^ma_iq^i\in\mathbb Z_{\ge 0}[q]$ and $1_{C_n}$ is the trivial representation of $C_n$.


With the above definition, Rao and Suk gave a general definition of group sieving.

\begin{definition}[\cite{MR4064826}, Definition 2.7]
 Let $G$ be a finite group and let $\rho_1,\dots,\rho_k$ be representations of $G$ over $\mathbb C$ which generate $\text{Rep}(G)$ as a ring. Let $X$ be a finite set carrying an action of $G$, $X(q_1,\dots,q_k)\in \mathbb{Z}[q_1,\dots,q_k]$. We say that the quadruple 
 $$(X,\ G,\ (\rho_1,\dots,\rho_k),\ X(q_1,\dots,q_k))$$ 
 exhibits \textbf{$G$-sieving} if and only if $[\mathbb{C}[X]]=[X(\rho_1,\dots,\rho_k)]$ in $\text{Rep}(G)$.
\end{definition}

Let $n$ be an odd integer. Let $I_2(n):=\langle r,s|r^n=s^2=1,rs=sr^{-1}\rangle$ be the dihedral group $I_2(n)$ of order $2n$. It is easy to see that the generating representations for $\text{Rep}(I_2(n))$ (see \cite[Section 3]{MR4064826}) are:
$$z_1:r\rightarrow \begin{pmatrix}
\text{cos}\frac{2\pi}{n}   &-\text{sin}\frac{2\pi}{n} \\
&\\
\text{sin}\frac{2\pi}{n}  & \text{cos}\frac{2\pi}{n}
\end{pmatrix},\ s\rightarrow \begin{pmatrix}
  0&1 \\
  &\\
  1&0
\end{pmatrix},$$
and
$$\quad-\text{det}:r\rightarrow -1, s\rightarrow 1.$$

In particular, we have the following.

\begin{definition}
     Let $n$ be an odd integer and $I_2(n)$ the dihedral group of order $2n$ which acts on a finite set X. Let $X(q,t)\in\mathbb Z[q,t]$. The quadruple 
     $$(X,\ I_2(n),\ \{\det,z_1\},\ X(q,t))$$ exhibits the dihedral sieving if  $\mathbb{C}[X]=X(\det,z_1)$ in $\text{Rep}(G)$.
\end{definition}

For the polynomial $X(q,t)$ involved in the dihedral sieving, there are two classes of polynomials appeared so far, one is the famous $q,t$-Catalan numbers $\text{Cat}_n(q,t)$ of Garsia and Haiman \cite{GarH} and the other one is the following. 
\begin{definition}\cite[Equation (2)]{MR3282645}
The generalized Fibonacci polynomials are a sequence $\{n\}_{q,t}$ of polynomials in $\mathbb N[q,t]$ defined inductively by
$$\{0\}_{q,t}=0,\quad \{1\}_{q,t}=1,$$
$$\{n+2\}_{q,t}=q\{n+1\}_{q,t}+t\{n\}_{q,t},$$
$$\{n\}!_{q,t}=\{n\}_{q,t}\{n-1\}_{q,t}\cdots\{1\}_{q,t},$$
and the Fibonomial coefficient is defined as
$$\begin{Bmatrix}
 n\\
k
\end{Bmatrix}_{q,t}=\frac{\{n\}!_{q,t}}{\{k\}!_{q,t}\{n-k\}!_{q,t}}.$$

\end{definition}

Note that the Fibonomial coefficient is a polynomial in $q$ and $t$ with nonnegative integral coefficients (see \cite[Theorem 5.2]{MR3282645}).     

In \cite{MR4064826}, Rao and Suk obtained some examples of the dihedral sieving on the following objects (with the corresponding $X(q,t)$) 
\begin{enumerate}
    \item the set of $k$-elements subsets of $[n]$ with $X(q,t)=\begin{Bmatrix}
 n\\
k
\end{Bmatrix}_{q,t}$,
    \item the set of $k$-elements multisubsets of $[n]$ with $X(q,t)=\begin{Bmatrix}
 n+k-1\\
k
\end{Bmatrix}_{q,t}$,
    \item the set of non-crossing partitions of then $n$-gon with 
    $$X(q,t)=\frac{1}{\{n+1\}_{q,t}}\begin{Bmatrix}
 2n\\
n
\end{Bmatrix}_{q,t},$$
    \item the set of non-crossing partitions of then $n$-gon using $n-k$ blocks with 
        $$X(q,t)=\frac{1}{\{n\}_{q,t}}\begin{Bmatrix}
 n\\
k
\end{Bmatrix}_{q,t}\begin{Bmatrix}
 n\\
k+1
\end{Bmatrix}_{q,t},$$
    \item the set of triangulations of the $n$-gon with $X(q,t)=(qt)^{\binom{n-2}{2}}\text{Cat}_{n-2}(q,t)$.
\end{enumerate}

In \cite{MR4459983}, the result (5) of Rao and Suk was generalized to $k$-angulations of an $n$-gon (together with some algebraic generalizations).

Under a natural action of the dihedral group $I_2(n-1)$ (for even $n$) on $X_n$, we prove a dihedral sieving result on $X_n$ in Section 4. Our result involves the polynomial $X(q,t)=\frac{1}{\{n\}_{q,t}}\begin{Bmatrix}
 n\\
k
\end{Bmatrix}_{q,t}\begin{Bmatrix}
 n\\
k+1
\end{Bmatrix}_{q,t}$.

The following sections are organized as follows. In Section 2, we define a statistic on the set $X_n$ whose generating function is $\text{Cat}_n(q)$. In section 3, we  answer the problem of Kim. In Section 4, we prove a dihedral sieving result for $X_n$. We end this paper with some concluding remarks in Section 5.

\bigskip





\section{A statistic on noncrossing (1,2)-configurations}
	In this section, we answer Problem \ref{Thiel'sproblem}. We would introduce a statistic on  $X_n$ and prove that its generating function is the $q$-Catalan number $\text{Cat}_n(q)$. We proceed by firstly recalling a result of MacMahon on Dyck paths and then constructing a bijection between noncrossing (1,2)-configurations and Dyck paths. 
        
	A {\it Dyck path}  is a lattice path from $(0,0)$ to $(n,n)$ that lies above (but may touch) the diagonal $y=x$. Let $\mathcal D_n$ denote the set of Dyck paths from $(0,0)$ to $(n,n)$. It is well-known that $|\mathcal D_n|=\frac{1}{n+1}\binom{2n}{n}$. Given $\pi\in \mathcal D_n$, let $\sigma(\pi)$ be a string resulting from the following algorithm.
    \begin{itemize}
        \item First initialize $\sigma(\pi)$ to be the empty string;
        \item Next, start at $(0,0)$, move along $\pi$ and add a $0$ to the end of $\sigma(\pi)$ every time an $N$ step is encountered, and add a $1$ to the end of $\sigma(\pi)$ every time an $E$ step is encountered.
    \end{itemize} 
    For example, the Dyck path $\pi$ in Figure \ref{fig3} satisfies $\sigma(\pi)=1110110001010010$. Let $\sigma=\sigma_1\sigma_2\cdots\sigma_n$ denote a string of length $n$. A {\it descent} of $\sigma$ is an integer $i$, $1\leq i\leq n-1$, for which $\sigma_i>\sigma_{i+1}$. Define the {\it major index statistic} of $\sigma$, denoted by $\text{maj}(\sigma)$, as the sum of the descents of $\sigma$ (we refer to \cite[Chapter 1]{Stan3} for more information about the descent and the major index statistic), i.e.,
	\begin{align*}
		\text{maj}
        (\sigma)={\sum_{\substack{i\\\sigma_i>\sigma_{i+1}}}}i.
	\end{align*}
    A {\it corner} $(a,b)$ of $\pi\in \mathcal D_n$ is a lattice point such that moving along $\pi$, encounter $(a-1,b)$, take a step to the east, encounter $(a,b)$, take a step to the north and then encounter $(a,b+1)$. Let $(a_1,b_1),(a_2,b_2),\dots,(a_\ell,b_\ell)$ be all the corners of $\pi$.
    By simple observation, we have
		\begin{equation*}
			\text{maj}(\sigma(\pi))=\sum_{i=1}^{\ell}(a_i+b_i).
		\end{equation*}
	\begin{lemma}{\rm(MacMahon \cite[p. 214]{macmahon2001combinatory})}\label{lem.Mac}
        For any positive integer $n$, we have
		\begin{equation*}
			\sum_{\pi\in \mathcal D_n}q^{\mathrm{maj}(\sigma(\pi))}=\frac{1}{[n+1]_q}\begin{bmatrix}
				2n\\
				n
			\end{bmatrix}_q.
		\end{equation*}
	\end{lemma}
    Now, we define a statistic on $X_n$.
	\begin{definition}
		For a noncrossing (1,2)-configuration
		$$x=\{\{a_1,b_1\},\dots\{a_\ell,b_\ell\},\{c_1\},\dots,\{c_s\}\}\in X_n ,$$
		we define a statistic {\it cwt} of 
		$x$ as
  $$\mathrm{cwt}(x)=\sum_{i=1}^{\ell}(a_i+b_i)+2\sum_{j=1}^{s}c_j.$$	
	\end{definition}
  Actually, the cwt statistic is a weight function on $X_n$, where the weight of a ball is two and the weight of an arc is one for each vertex. For example, for $x=\{\{1,3\},\{4,6\},\{5\},\{7\}\}\in X_8$, we have $\text{cwt}(x)=1+3+4+6+2\times(5+7)=38$. 
  
  Now, we prove the main result of this section.
	\begin{theorem}
		For any positive integer $n$, we have $$\sum_{x\in X_n}q^{\mathrm{cwt}(x)}=\frac{1}{[n+1]_q}\begin{bmatrix}
			2n\\
			n
		\end{bmatrix}_q.$$
	\end{theorem}
	\begin{proof}
		Define a {\it noncrossing pair sequence} over $[n-1]\times [n-1]$ to be a sequence of pairs $$(a_1,b_1)(a_2,b_2)\cdots(a_\ell,b_\ell)$$ such that 
        \begin{itemize}
        \item $0\leq\ell\leq n-1$, where $\ell=0$ means that this sequence contains no elements and we use $\emptyset$ to denote it,
        \item $1\leq a_i\leq b_i\leq n-1$ for $1\leq i\leq \ell$, 
        \item $a_1<a_2<\cdots<a_\ell$ , 
        \item  $\{a_i,b_i\}\cap \{a_j,b_j\}=\emptyset$ (for $i\neq j$ and $1 \le i,j \le \ell$) and it does not contain crossing pairs, i.e., pairs $(a_i,b_i)$ and $(a_j,b_j)$ with $ a_i< a_j<b_i<b_j$. 
        \end{itemize}
	We denote by ${NPS}_n$ the set of all noncrossing pair sequences over $[n-1]\times [n-1]$. 	
		It is clear that a noncrossing pair sequence $(a_1,b_1)(a_2,b_2)\cdots(a_\ell,b_\ell)$ over $[n-1]\times [n-1]$ can present a noncrossing (1,2)-configuration $\{\{a_1,b_1\},\{a_2,b_2\},\dots,\{a_\ell,b_\ell\}\}$. Conversely, given a noncrossing (1,2)-configuration, we can also obtian a noncrossing pair sequence by the following steps. Firstly, we regard an arc $\{a,b\}$ as a pair $(\text{min}\{a,b\},\text{max}\{a,b\})$ and regard a ball $\{c\}$ as a pair $(c,c)$. Secondly, we arrange these pairs in ascending order from the first position to obtain a pair sequence in ${NPS}_n$. In this way, there is a bijection $\varphi_1$ from $X_n$ to ${NPS}_n$. 

            Define an {\it ordered pair  sequence} over $[n-1]\times [n-1]$ to be a sequence of pairs $$(a_1,b_1)(a_2,b_2)\cdots(a_\ell,b_\ell)$$ such that 
        \begin{itemize}
        \item $0\leq\ell\leq n-1$, where $\ell=0$ means that this sequence contains no elements and we use $\emptyset$ to denote it,
        \item $1\leq a_i\leq b_i\leq n-1$ for $1\leq i\leq \ell$,
        \item $a_1<a_2<\cdots<a_\ell$ and $b_1<b_2<\cdots<b_\ell$,
        \item  $\{a_i,b_i\}\cap \{a_j,b_j\}=\emptyset$ for $i\neq j$ and $1\leq i,j\leq n-1$. 
        \end{itemize}
	We denote by ${OPS}_n$ the set of all ordered pair  sequences over $[n-1]\times [n-1]$.
		Given an ordered pair sequence of ${OPS}_n$, we can regard all elements of this sequence as corners of a Dyck path. Conversely, given $\pi\in \mathcal D_n$, all the corners of $\pi$ can be orderly arranged into an ordered pair sequence in ${OPS}_n$ (arrange these corners in ascending order from the first position).
		 In this way, there is a bijection $\varphi_2$ from $\mathcal D_n$ to ${OPS}_n$. 
        \begin{figure}
     \centering
     \includegraphics[width=0.75\linewidth]{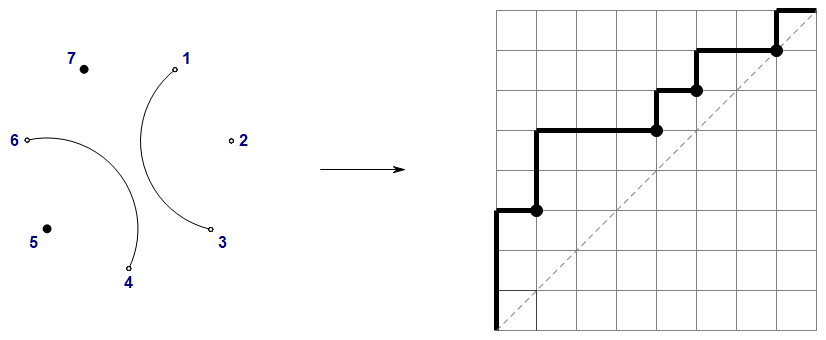}
     \caption{A noncrossing (1,2)-configuration $x$ and its image $\pi=\varphi_2^{-1}(\varphi(\varphi_1(x)))$ which has corners $(1,3)$, $(4,5)$, $(5,6)$ and $(7,7)$.}
     \label{fig3}
         \end{figure}
  
		Now, we construct a map $\varphi$ from ${NPS}_n$ to ${OPS}_n$:
		\begin{align*}
			\varphi: {NPS}_n&\to {OPS}_n,\\
			(a_1,b_1)(a_2,b_2)\cdots(a_\ell,b_\ell)&\mapsto(a_1,b_{i_1})(a_2,b_{i_2})\cdots(a_\ell,b_{i_\ell}),
		\end{align*}  
		where $\{b_{i_1},b_{i_2},\dots,b_{i_\ell}\}=\{b_1,b_2\dots,b_\ell\}$, $b_{i_1}<b_{i_2}<\cdots<b_{i_\ell}$ and $\varphi(\emptyset)=\emptyset$. .
		As $|{NPS}_n|=|X_n|$, $|{OPS}_n|=|\mathcal D_n|$, and $|X_n|=|\mathcal D_n|=\frac{1}{n+1}\binom{2n}{n}$, proving $\varphi$ bijective is equivalent to proving it surjective. Now, given a nonempty ordered pair sequence $(a_1,b_1)(a_2,b_2)\dots(a_\ell,b_\ell)$, we use the following algorithm to find its preimage. Firstly, we match the same elements of $\{a_1,a_2,\dots,a_\ell\}$ and $\{b_1,b_2,\dots,b_\ell\}$ to obtain several pairs $(a_i,a_i)$ which do not cross each other . If there are no elements left, we are done, i.e., 
		$$\varphi((a_1,a_1)(a_2,a_2)\dots(a_\ell,a_\ell)) =(a_1,a_1)(a_2,a_2)\dots(a_\ell,a_\ell).$$ 
		Otherwise, we write the remaining elements as $\{a_{i_1},a_{i_2},\dots,a_{i_s}\}$ and $\{b_{j_1},b_{j_2},\dots,b_{j_s}\}$ such that $a_{i_1}<a_{i_2}<\cdots<a_{i_s}$ and $b_{j_1}<b_{j_2}<\cdots<b_{j_s}$. Next, we match $b_{j_1}$ with some element of $\{a_{i_1},a_{i_2},\dots,a_{i_s}\}$. Clearly, $b_{j_1}>a_{i_1}$. We distinguish two cases:
		
		{\bf Case 1:} $b_{j_1}>a_{i_s}$. We match $b_{j_1}$ with $a_{i_s}$ and obtain a pair $(a_{i_s},b_{j_1})$. Clearly, $(a_{i_s},b_{j_1})$ does not cross $(a_i,a_i)$. Since $a_{i_p}<a_{i_s}<b_{j_1}<b_{i_q}$ for $p<s$ and $q>1$, $(a_{i_s},b_{j_1})$ will not cross any pair obtained from other remaining elements.
		
		{\bf Case 2:} $a_{i_k}<b_{j_1}<a_{i_{k+1}}$ for $1\leq k\leq s-1$. We match $b_{j_1}$ with $a_{i_k}$ and obtain a pair $(a_{i_k},b_{j_1})$. Clearly, $(a_{i_k},b_{j_1})$ does not cross $(a_i,a_i)$. Since $a_{i_p}<a_{i_k}<b_{j_1}<a_{i_q}$ for $p<k<q$ and $a_{i_k}<b_{j_1}<b_{i_m}$ for $m>1$, $(a_{i_k},b_{j_1})$ will also not cross any pairs obtained from other remaining elements.
		
		we can repeat the above steps to match $b_{j_2}$, $b_{j_3}$, and so on. Finally, we arrange these pairs in ascending order from the first position and obtain a sequence 
		$$(a_1,b_{m_1})(a_2,b_{m_2})\cdots(a_\ell,b_{m_\ell}).$$
		Note that $(a_1,b_{m_1})(a_2,b_{m_2})\cdots(a_\ell,b_{m_\ell})\in {NPS}_n$ and $$\varphi((a_1,b_{m_1})(a_2,b_{m_2})\cdots(a_\ell,b_{m_\ell}))=(a_1,b_1)(a_2,b_2)\cdots(a_\ell,b_\ell).$$ Therefore, $\varphi$ is bijective.
		
		For $x\in X_n$, we have $\mathrm{cwt}(x)=\text{maj}\left(\sigma\left(\varphi_2^{-1}(\varphi(\varphi_1(x)))\right)\right)$. Since $\varphi_2^{-1}\circ\varphi\circ \varphi_1$ is a bijection from $X_n$ to $\mathcal D_n$, we have
		\begin{equation*}
			\sum_{x\in X_n}q^{\mathrm{cwt}(x)}=\sum_{\pi\in \mathcal D_n}q^{\text{maj}(\sigma(\pi))}.
		\end{equation*}
		By Lemma \ref{lem.Mac}, we complete the proof.
	\end{proof}

 \section{The equivalence of two sieving functions}
	In this section, we answer Problem \ref{prob2}. 
	\begin{theorem}
		We have 
		\begin{equation*}
			q^{\binom{n}{2}}\text{\bf fd}((FDR_n)_{i+j=n-1})\equiv \mathrm{Cat}_n(q)\pmod{q^{n-1}-1},
		\end{equation*}
	where 
	\begin{equation*}
		\text{\bf fd}((FDR_n)_{i+j=n-1})=\sum_{\substack{k,x,y\\2k+x+y=n-1}}\begin{bmatrix}
			n-1\\
			2k,x,y
		\end{bmatrix}_q\mathrm{Cat}_k(q)q^{k+\binom{x}{2}+\binom{y}{2}}.
	\end{equation*}
	\end{theorem}
	\begin{proof}
		Let $\zeta_{n-1}$ be a primitive $(n-1)$-th root of unity and for simplicity, we denote $$P_n(q)=q^{\binom{n}{2}}\text{\bf fd}((FDR_n)_{i+j=n-1}).$$
		It suffices to prove that $P_n(\zeta_{n-1}^\ell)=\text{Cat}_n(\zeta_{n-1}^\ell)$ for $0\leq\ell\leq n-2$.
  
		Let $m=\frac{n-1}{\text{gcd}(n-1,\ell)}$.
		It is easy to see
		\begin{align*}
			&\text{Cat}_n(\zeta_{n-1}^\ell)=\begin{cases}
				\frac{1}{n+1}\binom{2n}{n}=\frac{1}{2n+1}\binom{2n+1}{n},&\text{ if }m=1 ,\ \zeta_{n-1}^\ell=1,\\
                &\\
				\binom{n}{\frac{n-1}{2}},&\text{ if }m=2,\ \zeta_{n-1}^\ell=-1,\\
                &\\
				\binom{\frac{2(n-1)}{m}}{\frac{n-1}{m}},&\text{ if }m>2,\\			
			\end{cases}\\
			&\begin{bmatrix}
				n-1\\
				2k,x,y
			\end{bmatrix}_{q=\zeta_{n-1}^\ell} =\begin{cases}
			\begin{pmatrix}
				\frac{n-1}{m}\\
				\frac{2k}{m},\frac{x}{m},\frac{y}{m}
			\end{pmatrix} ,& \text{ if } m|2k,\ m|x, \text{ and } m|y,\\
                &\\
			0,& \text{ otherwise. }
		\end{cases}
		\end{align*}
		We distinguish three cases:
		
		{\bf Case 1:} $m=1$ and $\zeta_{n-1}^\ell=1$. We have
		\begin{align*}
			P_n(1)&=\sum_{\substack{k,x,y\\2k+x+y=n-1}}\binom{n-1}{2k,x,y}\frac{1}{k+1}\binom{2k}{k}\\
				&=\sum_{\substack{k,x,y\\2k+x+y=n-1}}\binom{n-1}{k,k,x,y}\frac{1}{k+1}\\
				&=\frac{1}{n}\sum_{\substack{k,x,y\\2k+x+y=n-1}}\binom{n}{k,k+1,x,y}.\\
		\end{align*}
		Considering the coefficient (say $e_1$) of the term $z^{n-1}$ in $(z^2+z+z+1)^n$ and the coefficient (say $e_2$) of the term $z^n$ in $(z+1)^{2n+1}$, since $(2n+1)(z^2+z+z+1)^n=((z+1)^{2n+1})'$,  we have
		\begin{align*}
                &(2n+1)e_1=ne_2,\\
			&e_1=\sum_{\substack{k,x,y\\2k+x+y=n-1}}\binom{n}{k,k+1,x,y},\\
			&e_2=\binom{2n+1}{n},
		\end{align*}
		which implies that $P_n(1)=\text{Cat}_n(1)$.
		
		{\bf Case 2:} $m=2$ and $\zeta_{n-1}^\ell=-1$. When $2k+x+y=n-1$, $2|x$ and $2|y$, we have 
		\begin{equation*}
			\binom{n}{2}+k+\binom{x}{2}+\binom{y}{2}=\frac{(n-1)^2}{2}+\frac{x^2}{2}+\frac{y^2}{2}+2k\equiv 0\pmod 2,
		\end{equation*}
	and therefore
	\begin{equation*}
		q^{\binom{n}{2}}q^{k+\binom{x}{2}+\binom{y}{2}}|_{q=-1}=1.
	\end{equation*}
	Thus, we have 
	\begin{align*}
		P_n(-1)&=\sum_{\substack{k,x,y\\2|x,2|y\\2k+x+y=n-1}}\begin{pmatrix}
			\frac{n-1}{2}\\
			k,\frac{x}{2},\frac{y}{2}
		\end{pmatrix}\text{Cat}_k(-1)\\
	&=\sum_{\substack{k,x,y\\2|x,2|y\\k+\frac{x}{2}+\frac{y}{2}=\frac{n-1}{2}}}\begin{pmatrix}
		\frac{n-1}{2}\\
		k,\frac{x}{2},\frac{y}{2}
	\end{pmatrix}\text{Cat}_k(-1)\\
		&=\sum_{\substack{a,b,c\\2a+b+c=\frac{n-1}{2}}}\begin{pmatrix}
			\frac{n-1}{2}\\
			a,a,b,c
		\end{pmatrix}+
		\sum_{\substack{a,b,c\\2a-1+b+c=\frac{n-1}{2}}}\begin{pmatrix}
			\frac{n-1}{2}\\
			a,a-1,b,c
		\end{pmatrix}.
	\end{align*}
	Considering the coefficient (say $e_3$ and $e_4$) of the term $z^{\frac{n+1}{2}}$ in $(z^2+z+z+1)^{\frac{n-1}{2}}(z+1)$ and $(z+1)^n$, we have
	\begin{align*}
		e_3&=e_4,\\
		e_3&=\sum_{\substack{a,b,c\\2a+b+c=\frac{n-1}{2}}}\begin{pmatrix}
			\frac{n-1}{2}\\
			a,a,b,c
		\end{pmatrix}+
	\sum_{\substack{a,b,c\\2a+b+c=\frac{n+1}{2}}}\begin{pmatrix}
		\frac{n-1}{2}\\
		a,a-1,b,c
	\end{pmatrix}\\
	&=\sum_{\substack{a,b,c\\2a+b+c=\frac{n-1}{2}}}\begin{pmatrix}
		\frac{n-1}{2}\\
		a,a,b,c
	\end{pmatrix}+
	\sum_{\substack{a,b,c\\2a-1+b+c=\frac{n-1}{2}}}\begin{pmatrix}
		\frac{n-1}{2}\\
		a,a-1,b,c
	\end{pmatrix},\\
	e_4&=\binom{n}{\frac{n+1}{2}},
	\end{align*}
	which implies that $P_n(-1)=\text{Cat}_n(-1)$.
	
	{\bf Case 3:} $m>2$. When $m|2k$, we have 
	\begin{equation*}
		\text{Cat}_k(\zeta_{n-1}^\ell)=\begin{cases}
			\binom{\frac{2k}{m}}{\frac{k}{m}},&\text{ if }m|k,\\
			0,&\text{otherwise.}
		\end{cases}
	\end{equation*}
	When $2\not|m$, $m|k$, $m|x$, $m|y$ and $2k+x+y=n-1$, it is clearly that \begin{equation*}
		\binom{n}{2}+k+\binom{x}{2}+\binom{y}{2}\equiv 0\pmod m.
	\end{equation*}
	When $2|m$, $m|k$, $m|x$, $m|y$ and $2k+x+y=n-1$, we have 
	\begin{equation*}
		\binom{n}{2}+k+\binom{x}{2}+\binom{y}{2}=\frac{(n-1)^2}{2}+\frac{x^2}{2}+\frac{y^2}{2}+2k\equiv 0\pmod m.
	\end{equation*}
	Therefore, in any case above, we have 
	\begin{equation*}
		q^{\binom{n}{2}}q^{k+\binom{x}{2}+\binom{y}{2}}|_{q=\zeta_{n-1}^\ell}=1.
	\end{equation*}
	Thus, we have 
	\begin{align*}
		P_n(\zeta_{n-1}^\ell)&=\sum_{\substack{k,x,y\\m|k,m|x,m|y\\2k+x+y=n-1}}\begin{pmatrix}
			\frac{n-1}{m}\\
			\frac{2k}{m},\frac{x}{m},\frac{y}{m}
		\end{pmatrix}\binom{\frac{2k}{m}}{\frac{k}{m}}\\
		&=\sum_{\substack{k,x,y\\m|k,m|x,m|y\\\frac{2k}{m}+\frac{x}{m}+\frac{y}{m}=\frac{n-1}{m}}}\begin{pmatrix}
			\frac{n-1}{m}\\
			\frac{k}{m},\frac{k}{m},\frac{x}{m},\frac{y}{m}
		\end{pmatrix}\\
		&=\sum_{\substack{a,b,c\\2a+b+c=\frac{n-1}{m}}}\begin{pmatrix}
			\frac{n-1}{m}\\
			a,a,b,c
		\end{pmatrix}.
	\end{align*}
	Considering the coefficient (say $e_5$ and $e_6$) of the term $z^{\frac{n-1}{m}}$ in $(z^2+z+z+1)^{\frac{n-1}{m}}$ and $(z+1)^{\frac{2(n-1)}{m}}$, we have
	\begin{align*}
		e_5&=e_6,\\
		e_5&=
		\sum_{\substack{a,b,c\\2a+b+c=\frac{n-1}{2}}}\begin{pmatrix}
			\frac{n-1}{2}\\
			a,a,b,c
		\end{pmatrix},\\
	e_6&=\binom{\frac{2(n-1)}{m}}{\frac{n-1}{m}}.
	\end{align*}
	which implies that $P_n(\zeta_{n-1}^\ell)=\text{Cat}_n(\zeta_{n-1}^\ell)$.
	\end{proof}

\section{Dihedral sieving on noncrossing (1,2)-configurations}
	In this section, we prove a dihedral sieving result on $X_n$. 
 
    Let  $\langle \tau \rangle$ be a cyclic group of order $2$ such that $\tau$ acts on $X_n$ by flipping the diagrams left and right, i.e., $\tau:[n-1]\rightarrow[n-1],\ i\mapsto n-i$  (see Figure \ref{fig2} for an example). Then we have the following cyclic sieving result.
		\begin{lemma}\label{lem4.1}
		Let $X_n$ be the set of noncrossing  (1,2)-configurations of $[n-1]$. Let 
		$$X(q):=\mathrm{Cat}_n(q)=\frac{1}{[n+1]_q}\begin{bmatrix}
			2n\\
			n
		\end{bmatrix}_q,$$
		then $$(X_n,\langle \tau \rangle,X(q))$$exhibits the cyclic sieving phenomenon.
	\end{lemma}

 \begin{figure}
     \centering
     \includegraphics{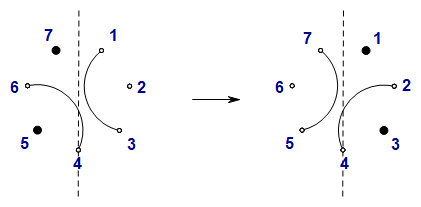}
     \caption{A noncrossing (1,2)-configuration $x$ and its reflection $\tau(x)=\{\{1\},\{2,4\},\{3\},\{5,7\}\}$.}
     \label{fig2}
 \end{figure}
	
\begin{proof}
Since $|X_n|=\frac{1}{n+1}\binom{2n}{n}=\mathrm{Cat}_n(1)$, it suffies to prove that
$$|\{x\in X_n:\tau(x)=x\}|=\text{Cat}_n(-1).$$
Clearly,
$$\text{Cat}_n(-1)=\frac{1}{[n+1] _{q=-1}}\begin{bmatrix}
2n\\n
\end{bmatrix}_{q=-1}=
\left\{\begin{matrix}
\binom{n}{\frac{n}{2}},&n\text{ is even,} \\
&\\
\binom{n}{\frac{n-1}{2}},&n\text{ is odd.}
\end{matrix}\right.$$
  For another side, we need to compute the number of elements of $X_n$ which are fixed under the action of $\tau$ (we call these noncrossing (1,2)-configurations fixed points). We proceed by induction on $n$.	
  Let $F(n-1)=|\{x\in X_n:\tau(x)=x\}|$ and $m=n-1$.  By simple observation, we can see $F(m+1)=2F(m)$ for $m$ is even. To compute the fixed points, we distinguish four cases: 
			
		{\bf Case 1:} Vertex $1$ is empty. Then the number of fixed points is $F(m-2)$.
		
		{\bf Case 2:} Vertex $1$ is a ball. Then the number of fixed points is $F(m-2)$.
		
		{\bf Case 3:} There is an arc connecting vertex $1$ and vertex $m$ for $m\geq 2$. Then the number of fixed points is $F(m-2)$.
		
		{\bf Case 4:} There is an arc connecting vertex $1$ and vertex $i$ for $2\leq i\leq \left \lfloor \frac{m}{2} \right \rfloor $ and $m\geq 4$. Then the number of fixed points is $\text{Cat}_{i-1}F(m-2i)$.
		
		Therefore, 
		$$F(m)=3F(m-2)+\sum_{i=2 }^{\left \lfloor \frac{m}{2} \right \rfloor}\text{Cat}_{i-1}F(m-2i),$$
		where $m\ge 4$. It can be directly verified that $F(0)=1,\ F(1)=2,\ F(2)=3$, and $F(3)=6$.
		
		Define $$f(x):=\sum\limits_{n=0}^{\infty}\text{Cat}_nx^n,$$
		where $\text{Cat}_0=1$. It is well-known that $f(x)=\frac{1-\sqrt{1-4x}}{2x}$.
		
		Let $B(2n)=\text{Cat}_n$ and $$g(x)=f(x^2)=\sum\limits_{n=0}^{\infty}B(2n)x^{2n}=\frac{1-\sqrt{1-4x^2}}{2x^2}$$
		where $g(0)=1$.
		Let $m=2k$. We have 
		\begin{equation}
			\begin{aligned}
				F(m)=F(2k)&=3F(2k-2)+\sum_{i=2 }^{k}\text{Cat}_{i-1}F(2k-2i)\\
				&=3F(2k-2)+\sum_{i=2 }^{k}B(2i-2)F(2k-2i)\\
				&=2F(2k-2)+\sum\limits_{i=0}^{k-1}B(2i)F(2k-2-2i).
			\end{aligned}
			\nonumber
		\end{equation}
		Let
		$$h(x):=\sum\limits_{k=0}^{\infty}F(2k)x^{2k}.$$
		Then
		\begin{align*}
			h(x)g(x)&=\sum_{k=0}^{\infty}\left(\sum_{n=0}^{k}B(2n)F(2k-2n)\right)x^{2k}\\
			&=\sum_{k=0}^{\infty}\left(F(2k+2)-2F(2k)\right)x^{2k}\\
			&=\sum_{k=0}^{\infty}F(2k+2)x^{2k}-2\sum_{k=0}^{\infty}F(2k)x^{2k}\\
			&=\frac{h(x)-1}{x^2}-2h(x).
		\end{align*}
		It follows that
		$$h(x)=\frac{2}{1-4x^2+\sqrt{1-4x^2}}.$$
		It is easy to see that $$h(x)=\sum\limits_{k=0}^{\infty}\binom{2k+1}{k}x^{2k}.$$
		Thus, we have 
		$$F(n-1)=
		\left\{\begin{matrix}
			\binom{n}{\frac{n}{2}},        &n\text{ is even,} \\
            &\\
			\binom{n}{\frac{n-1}{2}},
            &n\text{ is odd.}
		\end{matrix}\right.$$
		which agrees with the value of $\text{Cat}_n(-1)$.
	\end{proof}

\begin{lemma}{\cite[Equation (7)]{MR3282645}}\label{lem4.2}
    We make substitutions $q=X+Y,\quad t=-X/Y$ to simplify the Fibonomial coefficient
    $\begin{Bmatrix}
     n\\k
    \end{Bmatrix}_{q,t}$.Then we have
     $$X=\frac{q+\sqrt{q^2+4t}}{2},\quad Y=\frac{q-\sqrt{q^2+4t}}{2},$$
     $$\begin{Bmatrix}
 n\\
k
\end{Bmatrix}_{q,t}=Y^{k(n-k)}\begin{bmatrix}
n \\
k
\end{bmatrix}_{q=X/Y}.$$
\end{lemma}

 Next, we prove the main result of this section.
\begin{theorem}
        Let $X_n$ be the set of noncrossing $(1,2)$-configuration of $[n-1]$ and $X(q,t)=\frac{1}{\{n+1\}_{q,t}}\begin{Bmatrix}
 2n\\
n
\end{Bmatrix}_{q,t}$. Then the quadruple
$$(X_n,\ I_2(n-1),\ \{z_1,-\det\},\ X(q,t))$$
exhibits the dihedral sieving.    
\end{theorem}

\begin{proof}
    Let $\zeta_{n-1}$ be a primitive $(n-1)$-th root of unity. Note that the dihedral group $I_2(n-1)=\langle r,s|r^{n-1}=s^2=e,rs=sr^{-1}\rangle$, and let $C$ be the conjugacy classes of $I_2(n-1)$.

    By the definition of the  dihedral sieving, we shall show the character values of the permutation representation $\mathbb{C}[X]$. According to the cyclic sieving phenomenon of $(X_n,C_{n-1},\text{Cat}_n(q))$, we have
    $$\chi_{\mathbb{C}[X]}(r^\ell)=\text{Cat}_n(\zeta_{n-1}^\ell),$$
     where the input $r^\ell$ can refer to either the group element $r^\ell\in I_2(n-1)$ or the conjugacy class $\{r^\ell,r^{n-\ell-1}\}$ in $I_2(n-1)$. By Lemma \ref{lem4.1}, the number of fixed points under the action of $sr^\ell$ is $$\chi_{\mathbb{C}[X]}(sr^\ell)=\binom{n}{\frac{n}{2}}=\text{Cat}_n(\zeta_2)=\text{Cat}_n(-1).$$
    The remaining work is to prove
    $$\chi_{\mathbb{C}[X]}(s^ir^\ell)=\chi_{X(z_1,-\text{det})}(s^ir^\ell)$$
    for $i\in\{0,1\}$ and $\ell\in[0,n-2]$.
    By Lemma \ref{lem4.2}, we have
$$\begin{Bmatrix}
 n\\
k
\end{Bmatrix}_{q=\chi_{z_1}(C),t=\chi_{-\det}(C)}=Y^{k(n-k)}\begin{bmatrix}
n \\
k
\end{bmatrix}_{q=X/Y},$$
where $X,Y$ satisfy
$$X=\left\{\begin{matrix}
\frac{\chi_{z_1}(C)+\sqrt{\chi_{z_1}(C)^2+4\chi_{-\det}(C)} }{2}=\zeta^\ell_{n-1},&\text{if}\  C=\{r^\ell,r^{n-\ell-1}\}, \\
    &\\
  1, &\ \ \ \text{if}\  C=\{s,sr,sr^2,\dots\},
\end{matrix}\right.$$
$$Y=\left\{\begin{matrix}
\frac{z_1(C)-\sqrt{z_1(C)^2+4\chi_{-\det}(C)} }{2}=\zeta^{-\ell}_{n-1},      &\text{if} \ C=\{r^\ell,r^{n-\ell-1}\}, \\
&\\
  -1,&\ \ \ \text{if}\  C=\{s,sr,sr^2,\dots\},
\end{matrix}\right.$$
$$X/Y=\left\{\begin{matrix}
  \zeta^{2\ell}_{n-1},&\text{if}\  C=\{r^\ell,r^{n-\ell-1}\}, \\
  &\\
  -1,&\ \ \ \text{if} \ C=\{s,sr,sr^2,\dots\}.
\end{matrix}\right.$$
Then
\begin{align*}
    \{n+1\}_{\chi_{z_1}(C),\chi_{-\det}(C)}&=Y^n[n+1]_{q=X/Y}\\
    &=\left\{\begin{matrix}
 \zeta_{n-1}^{-\ell n}[n+1]_{\zeta_{n-1}^{2\ell}}, & \text{if}\  C=\{r^\ell,r^{n-\ell-1}\},\\
 &\\
  (-1)^n[n+1]_{\zeta_2},&\ \ \ \text{if}\  C=\{s,sr,sr^2,\dots\},
\end{matrix}\right.
\end{align*}
$$\begin{Bmatrix}
 2n\\
n
\end{Bmatrix}_{\chi_{z_1}(C),\chi_{-\det}(C)}=\left\{\begin{matrix}
  \zeta_{n-1}^{-\ell n^2}\begin{bmatrix}
 2n\\n
\end{bmatrix}_{q=\zeta^{2\ell}_{n-1}},&  \text{if}\  C=\{r^\ell,r^{n-\ell-1}\},\\
&\\
  (-1)^{n^2}\begin{bmatrix}
 2n\\n
\end{bmatrix}_{q=\zeta_2},&\ \ \ \text{if}\  C=\{s,sr,sr^2,\dots\}.
\end{matrix}\right.$$
Since $\text{gcd}(n-1,\ell)=\text{gcd}(n-1,2\ell)$, we have
\begin{align*}
    \frac{1}{[n+1]_{q=\zeta_{n-1}^{\ell}}}\begin{bmatrix}
 2n\\
n
\end{bmatrix}_{q=\zeta_{n-1}^\ell}=
\frac{1}{[n+1]_{q=\zeta_{n-1}^{2\ell}}}\begin{bmatrix}
 2n\\
n
\end{bmatrix}_{q=\zeta_{n-1}^{2\ell}}.
\end{align*}
Therefore,
$$X(\chi_{z_1}(C),\chi_{-\det}(C))=\left\{\begin{matrix}
  \frac{1}{\zeta_{n-1}^{-\ell n}[n+1]_{\zeta_{n-1}^\ell}}\zeta_{n-1}^{-\ell n^2}\begin{bmatrix}
 2n\\
n
\end{bmatrix}_{q=\zeta^\ell_{n-1}},& \text{if}\  C=\{r^\ell,r^{n-\ell-1}\},\\
&\\
  \frac{1}{[n+1]_{\zeta_2}}\begin{bmatrix}
 2n\\
n
\end{bmatrix}_{q=\zeta_2},&\text{if}\  C=\{s,sr,\dots\}.
\end{matrix}\right.$$
Since
$$\frac{\zeta_{n-1}^{-\ell n^2}}{\zeta_{n-1}^{-\ell n}}=\zeta_{n-1}^{-\ell n(n-1)}=1,$$
we have
\begin{align*}
    \chi_{X(z_1,-\det)}(C)&=X(\chi_{z_1}(C),\chi_{-\det}(C))\\
    &=\left\{\begin{matrix}
  \text{Cat}_n(q)|_{q=\zeta_{n-1}^\ell},&\text{if}\  C=\{r^\ell,r^{n-\ell-1}\}, \\
  &\\
  \text{Cat}_n(q)|_{q=\zeta_2}, &\ \ \ \text{if}\ C=\{s,sr,sr^2,\dots\}.
\end{matrix}\right.
\end{align*}
which agrees with $\chi_{\mathbb{C}[X]}(C)$.
This completes the proof.
\end{proof}

\section{Concluding remarks}

In this paper, we considered the CSP for a Catalan object.  In fact, Catalan numbers are special cases of rational Catalan numbers $\text{Cat}_{n,m}:=\frac{1}{n+m}\binom{n+m}{n}$; see, e.g., \cite{BR,HanZhang}. It is natural to consider the CSP for rational Catalan objects. Indeed, we shall study the CSP for zero-sum sequence over finite abelian groups \cite{HanZhang} and some preliminary results have been obtained. 

\bigskip

\subsection*{Acknowledgments}
The authors would like to thank their advisor,  Hanbin Zhang, for valuable suggestions on the manuscript.

\bibliographystyle{amsplain}


\end{document}